\documentclass[10pt, leqno]{amsart}
 \usepackage{graphics}
\usepackage{amsfonts}
\usepackage{graphicx,nicefrac}
\usepackage{mathrsfs,dsfont}
\usepackage{amsmath,amstext,amsthm,amssymb}

\newtheorem{Theorem}{Theorem}[section]
\newtheorem{Proposition}[Theorem]{Proposition}
\newtheorem{Lemma}[Theorem]{Lemma}

\theoremstyle{definition}

\title{A Property of Upper Level Sets of Lelong Numbers of Currents on $\mathbb{P}^2$}
\author{James J. Heffers}
\begin{document}

\maketitle

\author

\begin{abstract}
\noindent Let $T$ be a positive closed current of bidimension $(1,1)$ with unit mass on the complex projective space $\mathbb P^2$. For $\alpha > 2/5$ and $\beta = (2-2\alpha)/3$ we show that if $T$ has four point with Lelong number at least $\alpha$, the upper level set $E_{\beta}^+ (T)$ of points of $T$ with Lelong number strictly larger than $\beta$ is contained within a conic with the exception of at most one point.
\end{abstract}

\section{Introduction}

Let $T$ be a positive closed current of bidimension $(1,1)$ in $\mathbb{P}^2$ with unit mass, where
$$\|T\| := \int_{\mathbb{P}^2} T \wedge \omega =1$$

 \noindent and $\omega$ is the Fubini-Study form on $\mathbb{P}^2$.  We consider the following upper level sets of Lelong numbers $\nu( T, q)$ of the current $T$

\begin{center}

$E_{\alpha} (T) = \{q \in \mathbb{P}^2 \, | \, \nu(T,q) \geq \alpha \},$

$E_{\alpha}^+ (T) = \{q \in \mathbb{P}^2 \, | \, \nu(T,q) > \alpha \}.$

\end{center}

It has been shown by Siu \cite{Siu74} that $E_{\alpha} (T)$ is an analytic subvariety of dimension at most $1$ when $\alpha >0$.  We will continue the investigation of the geometric properties of these sets started by Coman in \cite{C06} and pursued in a more general setting by Coman, Guedj, and Truong in \cite{CG09} and \cite{CT15}.  It has been shown by Coman \cite[Theorem 1.1]{C06} that given a current $T$ as above and $\alpha \geq \frac{1}{2}$ then we can find a complex line $L$ such that all points $p$ satisfying the condition $\nu(T, p) >\alpha$ are contained in a complex line $L$ (with at most the exception of one point).  Simply put, there exists $L$ such that $|E_{\alpha}^+ (T) \backslash L| \leq 1$ (this result holds in general in $\mathbb{P}^n$ as well).  Coman then showed in  \cite[Theorem 1.2]{C06} an analogous theorem for conics, that given a current $T$ as above and $\alpha \geq \frac{2}{5}$, then there is a conic $C$ such that $|E_{\alpha}^+ (T)\backslash C|\leq 1$.  He then proceed to establish in \cite[Theorem  3.10]{C06} that given two points $q_1$ and $q_2$ with Lelong number $\nu(T,q_i) \geq \alpha > \frac{1}{2}$, and setting $\beta = (2-\alpha) /3 $, then there exists $L$ such that $|E_{\beta}^+ (T) \backslash L| \leq 1$, showing that  when we have the existence of two points with large Lelong numbers, we can find a line containing a larger upper level set.

 Our goal is to establish a result analogous to Coman's above mentioned result \cite[Theorem 3.10]{C06} for conics, i.e. to find $\beta$ in terms of $\alpha$ such that given a few points in $E_\alpha (T)$, we can find a conic that either contains $E_{\beta}^+(T)$ or at most omits one point of $E_{\beta}^+ (T)$.  Coman showed that we needed two points of ``large" Lelong number in his result, and that it fails if we have less than two such points.  Since two points uniquely define a complex line, one may suspect initially that we would need five points in general position with ``large" Lelong number to make an analogous result for conics, as five points in general position define a unique conic.  However it turns out that we only need four such points, and that the four points can be in any position.  Specifically, we want to prove the following:

\begin{Theorem}\label{T:mt1} Let $T$ be a positive closed current of bidimension $(1,1)$ in $\mathbb{P}^2$, $\alpha > 2/5$ and $\beta = \frac{2}{3} (1- \alpha)$.  Let $\{ q_{i} \}_{i=1}^{4}$ be points in $\mathbb{P} ^{2}$ such that $\nu (T, q_{i}) \geq \alpha$.  Then there exists a conic $C$ (possibly reducible) such that $| E_{\beta }^+ (T) \backslash C | \leq 1$.
\end{Theorem}

 After proving this, we will look at several examples to establish that each assumption is necessary.  Example 3.7 shows that we have situations where for any conic $C$, $|E_{\beta}^+(T)\backslash C| = 1$.  We then show in example 3.8 that our value of $\beta$ is sharp for this property.  Finally examples 3.9 and 3.10 show that if we only have three points of ``large" Lelong number (either in general position or collinear in each example respectively) then the conclusion fails to hold.

\smallskip

\noindent \textbf{Acknowledgments.}  The author would like to thank Professor Dan Coman for his support, insights, and suggestions.  Additionally the author thanks the referee for their comments and suggestions. 

\section{Preliminaries}

 In an attempt to try keep this paper self contained, we will review the tools pivotal to proving the result.  The following two theorems by Coman were mentioned in the previous section, but stated below for convenience.

\begin{Theorem}\cite[Theorem 1.1]{C06} \label{T:21} Let $T$ be a positive closed current of bidimension $(1,1)$ in $\mathbb{P}^n$.  If $\alpha \geq \frac{1}{2}$ then there exists a line $L$ such that $|E_{\alpha}^+ (T) \backslash L| \leq 1$.  Moreover, if $\alpha \geq 2/3$ then $E_{\alpha}^+ (T) \subset L$.
\end{Theorem}

\begin{Theorem}\cite[Theorem 1.2]{C06}\label{T:22} Let $T$ be a positive closed current of bidimension $(1,1)$ in $\mathbb{P}^2$.  If $\alpha \geq \frac{2}{5}$ then there exists a conic $C$ (possibly reducible) such that $|E_{\alpha}^+ (T) \backslash C| \leq 1$.
\end{Theorem}

 We will also need to use entire pluricomplex Green functions in the upcoming  result.  Pluricomplex Green functions were introduced and studied in bounded domains in \cite{D87} , \cite{K85}, \cite{L89}, and \cite{L83}. Special cases were considered in \cite{C00} and \cite{C02}.  Let $S = \{p_1,  \dots ,p_k\} \subset \mathbb{C}^n$, and let $u\in PSH(\mathbb{C}^n)\cap L_{loc}^{\infty}(\mathbb{C}^n \backslash S)$ be such that $u = -\infty$ when restricted to $S$.  Define $\gamma_u$ as follows

\begin{center}

$$\gamma_u := \limsup_{\|z\| \rightarrow +\infty} \frac{u(z)}{\log{\|z\|}}  .$$
\end{center}

\noindent If $\gamma_u$ is finite, we say $u$ has logarithmic growth. If in addition $u$ satisfies the Monge-Amp\`ere equation $(dd^c u)^n = 0$ away from $S$, then $u$ is an entire pluricomplex Green function.  If for $p_i\in S$ we have 

\begin{center}
$u(z) - \alpha \log \|z - p_i \| = O(1) \; \text{as} \; z\rightarrow p_i $
\end{center}

\noindent then $u$ has a logarithmic pole of weight $\alpha$ at $p_i$.  Further, let $\widetilde{E} (S) \subset PSH(\mathbb{C} ^n)\cap L_{loc}^{\infty}(\mathbb{C}^n \backslash S)$ be the class of plurisubharmonic functions that have logarithmic poles of weight one at the points of $S$ and logarithmic growth.  With this information, we have the following two propositions by Coman that we will need:

\begin{Proposition}\cite[Proposition 2.1]{C06}\label{P:21}  Let $S = \{p_1, \dots , p_k\} \subset \mathbb{C}^n$ and let $T$ be a positive closed current of bidimension $(l,l)$ on $\mathbb{P}^n$.  If $u\in PSH(\mathbb{C}^n)$ has logarithmic growth, it is locally bounded outside a finite set, and $u(z) \leq \alpha_i \log\|z - p_i\| + O(1)$ for $z$ near $p_i$, where $\alpha_i > 0$, $1\leq i \leq k$, then

\begin{center}
\begin{displaymath}
\sum_{i=1}^{k} \alpha_{i}^{l} \nu(T, p_i) \leq \gamma_u^l \|T\| \; .
\end{displaymath}
\end{center}

\end{Proposition}

We define $m_j(S) := \max \{ |S\cap C| : \text{$C$ an algebraic curve}, \deg C =j\}$, i.e. the maximum number of points of $S$ contained on a degree $j$ algebraic curve.

\begin{Proposition}\cite[Proposition 2.4.(i)]{C06}\label{P:22}  Let  $A\subset \mathbb{C}^2$ with $|A| = 7$, $m_1(A)\leq3$, $m_2(A) = 6$, and let $\Gamma$ be the conic such that $|A\cap \Gamma| = 6$.  Let $q\notin A\cup \Gamma $.  If $m_1(A\cup \{q\})\leq 3$, then there exists $u\in PSH(\mathbb{C}^2)$ with $\gamma_u = 3$ such that $u$ is locally bounded outside a finite set, and $u(z) \leq \log\|z - p\| + O(1)$ near each $p\in A\cup \{q\}$.
\end{Proposition}

We will also make use of the next proposition which follows easily from Demailly's regularization theorem \cite[Proposition 3.7]{D92}.  

\begin{Proposition} \label{P:25}  Let $R$ be a positive closed current of bidegree $(1,1)$ on $\mathbb{P}^2$, $\nu (R, x_i) > a_i$, $i = 1, \dots , N$ for $x_i \in \mathbb{P}^2$ and $a_i > 0$.  Then there exists a positive closed bidegree $(1,1)$ current $R'$ on $\mathbb{P}^2$ with analytic singularities such that $\|R'\| = \|R\|$, $\nu(R',x_i) > a_i$ for $i = 1, \dots , N$, and $\nu(R' , x) \leq \nu(R,x)$ for all $x\in \mathbb{P}^2$.  In particular, $R'$ is smooth in a neighborhood of every point where $R$ has $0$ Lelong number.
\end{Proposition}




\section{Proof of the Main Theorem}

 First we prove the following lemmas that will be quite useful to us in the upcoming proofs.  They show that for $T$, a positive closed current of bidimension $(1,1)$  on $\mathbb{P}^2$, $T$ cannot have small mass if the points of $T$ with large Lelong number have certain configurations.

\begin{Lemma}\label{L:31}  Let $T$ be a positive closed current of bidimension $(1,1)$ in $\mathbb{P}^2$, $\alpha > 2/5$ and $\beta = \frac{2}{3} (1- \alpha)$.  Assume that $\{ q_{i} \}_{i=1}^{4}$ are points in $\mathbb{P} ^{2}$ such that $\nu (T, q_{i}) \geq \alpha$ and $\{ p_{i} \}_{i=1}^{4}$ be points in $\mathbb{P} ^{2}$ such that $\nu (T, p_{i}) > \beta$, let $\{x_i\}_{i=1}^{8}$ be a relabeling of $\{q_i\}_{i=1}^{4} \cup \{p_i\}_{i=1}^{4}$.  Assume $x_1, \dots  ,x_4 \in L_1$, where $L_1$ is a complex line, and either

\begin{itemize}

\item[i)] there exist complex lines $L_2$ and $L_3$ such that $\{x_1, x_5, x_6\} \in L_2$ and $\{x_2, x_7, x_8\} \in L_3$, or

\item[ii)] there exists an irreducible conic $\Gamma$ such that $\{x_1, x_2, x_5, x_6, x_7, x_8\}\in \Gamma$.

\end{itemize}

\noindent Then $\|T\| >1$.

\end{Lemma}

\begin{proof}

Suppose for contradiction that $\|T\| \leq 1$.  Note that the current $ S: =T/\|T\|$ has mass 1, and if $\nu(T,x) > c$, then $\nu(S, x) >c$, so we may assume that $\|T\| = 1$.  

 (i)  By Siu's decomposition theorem \cite{Siu74}, the current $T$ can be decomposed as follows:

$$T = a [L_1 ] + b[L_2] + c[L_3] + R,$$

\smallskip

\noindent where $R$ is a positive closed current of bidimension $(1,1)$, i.e. bidegree $(1,1)$, on $\mathbb{P}^2$, $R$ has generic Lelong number $0$ along each $L_i$, and $0\leq a, b, c \leq 1$ are the generic Lelong numbers along $L_1, L_2, L_3$ respectively.  Thus we now have

$$ R = T - a[L_1] - b[L_2] - c[L_3].$$

\smallskip

Choose $\alpha ' $ such that $\alpha > \alpha ' > 2/5$ and $\nu(T, p_i) > \frac{2}{3}(1-\alpha ') = \beta ' > \beta$.  Let $\{x_i\}_{i=1}^{8}$, be as they are in the assumptions.  Using this new information, we have the following:

$$\nu(R, x_1) = \nu(T,x_1) - a - b,\quad \nu(R, x_2) = \nu(T,x_2) - a - c $$
$$\nu(R, x_3) = \nu(T,x_3) - a,\quad \nu(R, x_4) = \nu(T,x_4) - a, \quad \nu(R,x_5)=\nu(T,x_5) - b$$
$$\nu(R, x_6) = \nu(T,x_6) - b,\quad \nu(R, x_7) = \nu(T,x_7) - c, \quad \nu(R,x_8) = \nu(T,x_8) - c$$

\smallskip
\noindent which gives us that

$$\sum_{i=1}^8 \nu(R, x_i) > 4 \alpha ' + 4\beta ' - 4a -3b -3c.$$

By proposition 2.5, we have a current $R'$, such that $\|R'\| = \|R\|$, $R'$ preserves the above inequality, and $R'$ is smooth wherever $R$ has Lelong number $0$.  Since the set of singularities of $R'$ is analytic, and $R'$ is smooth at generic points of $L_i$,  \cite[Corollary 2.10]{D93} tells us that $R' \wedge [L_i]$, $i=1,2,3$ is well defined measures.  Let $S:=([L_1]+[L_2]+[L_3])$, and thus $R' \wedge S$ is well defined.  We now have

$$3(1-a-b-c) = \int_{\mathbb{P}^2} R'\wedge S \geq \sum_{i=1}^8 R'\wedge S(\{x_i\})$$

$$ \geq \sum_{i=1}^8 \nu(R', x_i) > 4 \alpha '+ 4\beta ' - 4a -3b -3c$$

\smallskip

\noindent where the first equality comes from \cite[Theorem 4.4]{FS95} and the second inequality comes from the comparison theorem for Lelong numbers \cite[Corollary 5.10]{D93}, since 

$$\int_{\mathbb{P}^2} R'\wedge S \geq \sum \nu(R'\wedge S, x_i) \geq \sum \nu(R',x_i) \nu(S,x_i)$$

\smallskip

\noindent and $\nu(S,x_i) \geq 1$.  So we now have

$$ 3(1-a-b-c)>   4 \alpha '+ 4\beta '- 4a -3b -3c \implies a >\frac{4\alpha' -1}{3}.$$

\smallskip

Consider now just the current $R_a = T - a[L_1]$, and $S_a = \frac{R_a}{1-a}$, note that $\|S_a\| = 1$ and for $x_i \notin L_1$ we have either

$$\nu(S_a, x_i) = \frac{\nu(R_a,x_i)}{1-a} > \frac{\alpha'}{1 - \frac{4\alpha' -1}{3}} = \frac{3\alpha'}{4-4\alpha'} > \frac{1}{2}$$

 \noindent or

$$\nu(S_a, x_i) = \frac{\nu(R_a,x_i)}{1-a} > \frac{\beta'}{1 - \frac{4\alpha' -1}{3}} = \frac{2(1-\alpha')}{4(1-\alpha')} = \frac{1}{2}$$

\smallskip

\noindent so by \cite[Theorem 1.1]{C06} (see Theorem \ref{T:21}), $m_1(\{x_5,x_6,x_7,x_8\}) \geq 3$, which is a contradiction since $m_1(\{x_5,x_6,x_7,x_8\}) =2$.  

(ii)  Let $b$ be the generic Lelong number of $\Gamma$.  We use the same argument as above, and consider the measures $R'\wedge [L_1]$ and $R'\wedge[\Gamma]$ to get

$$3(1-a-2b) = \int_{\mathbb{P}^2} R'\wedge[L_1] + \int_{\mathbb{P}^2} R'\wedge[\Gamma ]  \geq \sum_{i=1}^8 \nu(R', x_i) > 4 \alpha' + 4\beta' - 4a -6b$$
\smallskip

\noindent which again gives $a> \frac{4\alpha' - 1}{3}$.  Now considering $R_a$ gives us the same contradiction.

\end{proof}

If $L_1$ contains one, two or three of the points $q_i\in E_{\alpha} (T)$, then we can drop the assumptions (i) and (ii) of the previous lemma:

\begin{Lemma}\label{L:32} Let $T$ be a positive closed current of bidimension $(1,1)$ on $\mathbb{P}^2$, $\alpha > 2/5$, $\beta =\frac{2}{3}(1-\alpha)$, $\{q_i\}_{i=1}^{4}$ and $\{p_i\}_{i=1}^{4}$ be points in $\mathbb{P}^2$ such that $\nu(T,q_{i})\geq \alpha > 2/5$ and $\nu(T,p_{i})> \beta$.  Assume there exists a complex line $L$ containing either $\{q_1,q_2, p_1, p_2\}$, $\{q_1, p_1, p_2, p_3\}$, or  $\{q_1,q_2, q_3, p_1\}$ and the four points not on $L$ are in general position. Then $\|T\| > 1$. 
\end{Lemma}

\begin{proof} Arguing as we did at the start of the previous lemma, we may assume $\|T\|=1$.  We will show that we can construct a conic satisfying the hypothesis of Lemma \ref{L:31}, and then we are done as Lemma \ref{L:31} says $\|T\| > 1$.  Suppose $L$ is a complex line containing $\{p_1, p_2, q_1, q_2\}$, and we will let $B = \{q_3, q_4, p_3, p_4\}$.  Then by the hypothesis, $m_1(B) = 2$.    Let $\alpha '$ be such that $\alpha > \alpha ' > 2/5$ and $\nu (T, p_{i}) > \frac{2}{3} (1- \alpha ') > \beta$.  Note that either $m_1(\{p_1, p_3, p_4\}) = 2$ or $m_1(\{p_2,p_3,p_4\}) = 2$, and w.l.o.g. say that $p_1, p_3, p_4$ are in general position.  We will let $L_{jk}$ be the line containing $p_j$ and $p_k$, and consider the current given by 

$$R = \frac{5\alpha ' - 2}{15\alpha '}( [L_{13}]+[L_{14}] + [L_{34}]) + \frac{2}{5\alpha '} T $$

\noindent and note $\|{R}\| = 1 $.  We have the following inequalities:

$$\nu (R, q_i) \geq \frac{2}{5\alpha '} \alpha > \frac{2}{5} \;, i = 1,2,3,4$$ 

\noindent and

$$ \nu (R, p_{i}) > \frac{10\alpha ' - 4}{15\alpha '} + \frac{4-4\alpha '}{15\alpha '} = \frac{2}{5}\; ,i=1,3,4. $$

\smallskip

Thus by Coman \cite[Theorem 1.2]{C06} (see Theorem \ref{T:22}), there is a conic $\Gamma$ containing at least six of $\{ q_{i} \}_{i=1}^{4} \cup \{ p_1, p_3, p_4\}$.  Note that $\Gamma$ cannot contain all seven points, otherwise $L$ would be a component of $\Gamma$, which would mean that $\Gamma$ is a reducible conic and thus that $m_1(B) >2$ since the points off of $L$ must also be collinear.  Likewise, the point $\Gamma$ must omit is one of the points on $L$, i.e. it must omit one of $q_1, q_2$ or $p_1$.  If $\Gamma$ is irreducible, then we are done.  If not, then note $\Gamma$ must be a reducible conic consisting of two lines, say $\Gamma = L_1 \cup L_2$.  Since $\Gamma$ contains all four points of $B$, it must be the case that each line $L_i$ contains exactly two points of $B$ (since $m_1(B) = 2$), and as no points of $B$ are on $L$, we have that each $L_i$ also contains a point of $L\cap \Gamma$.  Finally note that since $\Gamma$ contains six points, $L_1$ and $L_2$ cannot share the same point on $L$, i.e. $L_1 \cap L_2,\cap L = \emptyset$.  So we now have all of the hypotheses of Lemma 3.1 satisfied, and thus $\|T\| > 1$, a contradiction.

If we have that $L$ contains $\{q_1, p_1, p_2, p_3\}$, and $B = \{q_2, q_3, q_4, p_4\}$ is such that $m_1(B) = 2$, then using the current given by

$$R = \frac{5\alpha ' - 2}{15\alpha '}( [L]+[L_{14}] + [L_{24}]) + \frac{2}{5\alpha '} T ,$$ 
\smallskip

\noindent we can argue as we did above to get a conic $\Gamma$ containing six of the points in $\{q_1, q_2, q_3, q_4, p_1, p_2, p_4\}$ satisfying the conditions of Lemma \ref{L:31}, and we are done.

 Finally if we have that $L$ contains $\{q_1, q_2, q_3, p_1\}$, and $B = \{q_4, p_2, p_3, p_4\}$ is such that $m_1(B) = 2$, then using the current given by

$$R = \frac{5\alpha ' - 2}{15\alpha '}( [L_{23}]+[L_{24}] + [L_{34}]) + \frac{2}{5\alpha '} T ,$$ 

\smallskip

\noindent we can argue as we did above to get a conic $\Gamma$ containing six of the points in $\{q_1, q_2, q_3, q_4, p_2, p_3, p_4\}$ satisfying the conditions of Lemma \ref{L:31}, and again, we are done.

\end{proof}

\begin{Lemma}\label{L:33} Let $T$ be a positive closed current of bidimension $(1,1)$ on $\mathbb{P}^2$, $\alpha > 2/5$, $\beta =\frac{2}{3}(1-\alpha)$, $\{q_i\}_{i=1}^{4}$ and $\{p_i\}_{i=1}^{5}$ be points in $\mathbb{P}^2$ such that $\nu(T,q_{i})\geq \alpha > 2/5$ and $\nu(T,p_{i})> \beta$.  Assume there exist three distinct complex lines $L_1$, $L_2$, and $L_3$ containing $\{q_1,q_2, q_3, p_1\}$, $\{q_1,q_4, p_2, p_3\}$, and $\{q_3,q_4, p_4, p_5\}$, respectively. Then $\|T\| > 1$. 
\end{Lemma}

\begin{proof}  Suppose for contradiction that $\|T\| = 1$.  We attack this situation in cases, depending on how the points $p_1$, $p_2$, $p_3$, $p_4$, $p_5$, $q_2$ (i.e. the points not on the intersections of the three lines) fall.  First note that $m_1(\{p_2, p_3, p_4, p_5\}) = 2$.  We now break this into cases.

\textit{Case} 1:  Suppose that $m_1(\{p_2, p_3, p_4, p_5, q_2\}) = 2$.  Then consider the points $q_1, q_2, p_3, p_4, p_5$, noting that they are in general position, so there is an irreducible conic $\gamma_1$ containing them.  Now consider the current $R = T - a[L_1] - b[L_2] - c[\gamma_1]$, where $0\leq a,b,c \leq1$ are the generic Lelong numbers of $T$ along $L_1, L_2, \gamma_1$ respectively.  Let $\alpha' \in (2/5,\alpha)$ be as before, i.e. $\nu(T, p_i) > \frac{2}{3}(1-\alpha ' )=\beta '>\beta $.  Then by using proposition 2.5 as we did in lemma \ref{L:31}, there is a current $R'$ such that $\|R'\| = \|R\|$, $R'$ maintains the same lower bounds, and \cite[Corollary 2.10]{D93} gives us that $R'\wedge [L_i]$, $R'\wedge[\gamma_1]$ are well defined measures.  Define $S:= ([L_1]+[L_2]+[\gamma_1]) $, and now we have

$$4(1-a-b-2c) = \int_{\mathbb{P}^2} R'\wedge S \geq \sum \nu(R', x_i) \nu(S, x_i) \geq$$ 
$$2\nu(R', q_2) + \nu(R', q_3) + \nu(R', q_4) + \sum_{i=1}^{5} \nu(R', p_i)+\nu(R', p_3)$$
$$> 4 \alpha '+ 6\beta ' - 4a -4b -6c.$$

\smallskip
Now using the above inequality we get

$$4 - 2c > 4\alpha ' +6\beta ' = 4\alpha ' + 4(1-\alpha ') = 4$$

\smallskip
\noindent which is a contradiction as $c\geq 0$.  We will use similar techniques to handle the remaining cases.

\textit{Case} 2:  We have $m_1(\{p_2, p_3, p_4, p_5, q_2\}) = 3$.  That means $q_2$ is on a line with two $p_i$, one of the $p_i$ is on $L_2$ and one on $L_3$, say w.l.o.g. $ m_1(\{q_2, p_2, p_4\}) = 3$.  

\textit{Case} 2a:  If $m_1(\{q_2, p_3, p_5\}) = 2$, then the same argument as above gets us to a contradiction.
\vspace{10pt}

\textit{Case} 2b:  We have $ m_1(\{q_2, p_3, p_5\}) = 3$, $m_1(\{q_2, p_2, p_4\}) = 3$ and also that $m_1(\{p_1, p_2, p_3, p_4, p_5\}) = 2$. We observe that this means $m_1(\{ q_2, q_4, p_1, p_2, p_5\}) = 2$, $   m_1(\{ q_2, q_4, p_1, p_3, p_4\}) = 2$ and there are irreducible conics $\gamma_1$ and $\gamma_2$ containing $\{ q_2, q_4, p_1, p_2, p_5\}$ and $\{ q_2, q_4, p_1, p_3, p_4\}$ respectively.  Define a current $R = T - a[\gamma_1] - b[\gamma_2]$, let $\alpha ' $ be as before, and then once again proposition 2.5 and \cite[Corollary 2.10]{D93} gives a current $R'$ such that $\|R'\| = 1 - 2a - 2b$ and

$$4(1-2a-2b) = \int_{\mathbb{P}^2} R'\wedge([\gamma_1]+[\gamma_2]) \geq$$ 
$$2\nu(R', q_2) + 2\nu(R', q_4) + \nu(R', p_1) + \sum_{i=1}^{5} \nu(R', p_i)$$
$$> 4 \alpha '+ 6\beta ' - 8a -8b$$
$$\implies 4>4\alpha ' + 6\beta ' = 4$$

\smallskip
\noindent again giving us a contradiction.

\textit{Case} 2c:  We have $ m_1(\{q_2, p_3, p_5\}) = 3 = m_1(\{q_2, p_2, p_4\})$, and $m_1(\{p_1,\dots ,p_5\})=3$, so either $m_1(\{p_1, p_2, p_5\}) = 3$ or $m_1(\{p_1, p_3, p_4\})=3$.  Suppose $m_1(\{p_1, p_2, p_5\}) = 3$ and $m_1(\{p_1, p_3, p_4\})=2$ then note $ m_1(\{ q_2, q_4, p_1, p_3, p_4\})=2$ and there is an irreducible conic $\gamma_1$ containing  $\{ q_2, q_4, p_1, p_3, p_4\}$.  Let $l_1$ be the line containing $p_1, p_2, p_5$ and $l_2$ be the line containing $q_2, q_4$.  Note that by construction, none of the $p_i$ can fall on $l_2$ and $p_2,p_5 \notin \gamma_1$, otherwise either $L_2$ or $L_3$ would be a component of $\gamma_1$, which cannot be as $\gamma_1$ is irreducible.  Define a current $R = T - a[\gamma_1] - b[l_1]-c[l_2]$, let $\alpha ' $ be as before, and then proposition 2.5 and \cite[Corollary 2.10]{D93} gives a current $R'$ such that $\|R'\| = 1 - 2a - b-c$ and

$$4(1-2a-b-c) = \int_{\mathbb{P}^2} R'\wedge([\gamma_1]+[l_1]+[l_2]) \geq$$ 
$$2\nu(R', q_2) + 2\nu(R', q_4) + \nu(R', p_1) + \sum_{i=1}^{5} \nu(R', p_i)$$
$$> 4 \alpha '+ 6\beta ' - 8a -4b - 4c$$
$$\implies 4>4\alpha ' + 6\beta ' = 4$$

\smallskip

\noindent again giving us a contradiction.  If instead  $ m_1(\{q_2, p_3, p_5\}) = 3 = m_1(\{q_2, p_2, p_4\})$, $m_1(\{p_1, p_2, p_5\}) = 2$ and $m_1(\{p_1, p_3, p_4\})=3$ , a similar argument gives us a contradiction.

\textit{Case} 2d:  Finally $ m_1(\{q_2, p_3, p_5\}) = 3 = m_1(\{q_2, p_2, p_4\})$, $m_1(\{p_1, p_2, p_5\}) = 3$ and $m_1(\{p_1, p_3, p_4\})=3$.  Consider the seven points subset $\{q_2, q_4, p_1, p_2, p_3, p_4, p_5\}$, and note that we have $\{q_4, p_2, p_3\}\in L_2$, $\{q_4, p_4, p_5\}\in L_3$, and we also have lines $l_1, l_2, l_3, l_4$ containing $\{q_2, p_3, p_5\}$, $\{q_2, p_2, p_4\}$,  $\{p_1, p_3, p_4\}$, and  $\{p_1, p_2, p_5\}$ respectively.  Note that $m_2(\{q_2, q_4, p_1, p_2, p_3, p_4, p_5\}) = 5$, so we can apply Coman's result \cite[Proposition 2.3]{C06}, so there exists an entire pluricomplex Green function $u$ with $\gamma_u = 4$, and $u$ has weight two logarithmic poles and three of the seven points, and weight one at the remaining four.  First note that we cannot have weight two poles at both $q_2$ and $q_4$, for if we do, then we also have a weight two pole at say $p_1$, and \cite[Proposition 2.1]{C06} gives us that

$$4 = \gamma_u \|T\| \geq  2\nu(T, q_2) + 2\nu(T, q_4) + 2\nu(T,p_1) + \sum_{i=2}^{5} \nu(T, p_i) > 4\alpha + 6\beta =4$$

\smallskip

\noindent a contradiction.  So since $u$ cannot have a double pole at both $q_2$ and $q_4$, at least one of the $l_i$ or $L_i$ will have two points such that $u$ double poles at both points and a third where $u$ has a single pole, say w.l.o.g. we have $l_1$ with this property, where $u$ has double poles at $x_1, x_2 \in l_1$and has a single pole at $x_3\in l_1$.  But now applying \cite[Proposition 2.1]{C06}, we get

$$4 \geq \int_{\mathbb{C}^2} [l_1]\wedge dd^c u \geq  2\nu([l_1], x_1) + 2\nu([l_1], x_2) + \nu([l_1],x_3) =2 + 2 +1 =5$$

\smallskip

 \noindent an obvious contradiction.  However now we have ruled out all of the possible ways in which $p_1$, $p_2$, $p_3$, $p_4$, $p_5$, $q_2$ fall, thus it must be the case that$\|T\|>1$.  \end{proof}

  We now prove the main result.  This is done by proving a few propositions which consider the various cases that can occur depending on how the four points are positioned.  For the remainder of this section, assume that $T$ is a positive closed current of bidimension $(1,1)$ on $\mathbb{P}^2$ with $\|T\|=1$.

\begin{Proposition}\label{P:31} Let  $\{ q_{i} \}_{i=1}^{4}$ be points in $\mathbb{P} ^{2}$ such that they are in general position and $\nu (T, q_{i}) \geq \alpha > 2/5$.  Let $\beta = \frac{2}{3} (1- \alpha)$.  Then there exists a conic $C$ (possibly reducible) such that $| E_{\beta }^+ (T) \backslash C | \leq 1$.
\end{Proposition}

\begin{proof} Let $\{ q_{i} \}_{i=1}^{4}$, be as above and let $p_1 \in E_{\beta}^+ (T)$, $p_1 \neq q_i$ (noting that if no such $p_1$ exists then we are done). Since the ${q_i}$ are in general position, we let $\Gamma _{1} $ be the unique conic defined by the $ q_{i}$ and ${p_1}$.  If $\Gamma_{1}$ satisfies the conclusion, then we are done.  If not then we can find two points, $p_2$ and $p_3$ such that $p_2, p_3 \in E_{\beta}^+ (T)\backslash \Gamma_1$.  Let $\alpha '$ be such that $\alpha > \alpha ' > 2/5$ and $\nu (T, p_{i}) > \frac{2}{3} (1- \alpha ') > \beta$.  If the $p_i$ are in general position, we will let $L_{jk}$ be the line containing $p_j$ and $p_k$.  Define a current R as follows:

$$R = \frac{5\alpha ' - 2}{15\alpha '}\sum_{1\leq j<k\leq 3} [L_{jk}] + \frac{2}{5\alpha '} T $$ 

\smallskip

\noindent and note $\|{R}\| = 1 $.  We have the following inequalities:

$$\nu (R, q_i) > \frac{2}{5\alpha '} \alpha > \frac{2}{5} $$ 

\noindent and

$$ \nu (R, p_{i}) > \frac{10\alpha ' - 4}{15\alpha '} + \frac{4-4\alpha '}{15\alpha '} = \frac{2}{5}. $$

\smallskip

If instead the $p_i$ are all on a line $L$, then we use the current

$$R = \frac{5\alpha ' - 2}{5\alpha '} [L] + \frac{2}{5\alpha '} T$$ 

\smallskip

\noindent and get the same inequalities as above.  In either case, by \cite[Theorem 1.2]{C06}, there is a conic $\Gamma_2$ containing at least six of the $\{ q_{i} \}_{i=1}^{4} \cup \{ p_{i} \}_{i=1}^{3}$.  As $\Gamma_1$ is uniquely defined by the $q_i$ and $p_1$, $\Gamma_2$ must omit one of the seven points, and the point omitted must be one of the $q_i$ or $p_1$, else $\Gamma_1 = \Gamma_2$, which means one or both of $p_2, p_3$ would be on $\Gamma_1$, which is a contradiction.  If $\Gamma_2$ satisfies the conclusion, then we are done.  So suppose $\Gamma_2$ does not satisfy the conclusion of our proposition, and then there is $p_4\in  E_{\beta}^+ (T)\backslash \Gamma_2$.

 We will let $A = \{q_i\}_{i=1}^{4} \cup \{p_i\}_{i=1}^{3} $, and we will note that $|A| = 7$, $m_2 (A) = 6$, $|A\cap \Gamma_2 | = 6$ and $p_4\notin A\cup \Gamma_2$.  We will make use of these observations shortly.  Define $S = A\cup\{p_4\}$.  We now consider the following possibilities for $S$: $m_1(S) \leq 3$, $m_1(S) = 4$, and $m_1(S)\geq 5$.

Suppose $m_1(S)\leq 3$.  Then this means that $m_1(A)\leq 3$ and by the above observations about $A$, we can apply  \cite[Proposition 2.4.(i)]{C06} (see Proposition 2.4), i.e. there exists $u\in PSH(\mathbb{C}^2)$ such that $\gamma_u = 3$, $u$ is locally bounded outside of a finite set, and $u$ has logarithmic poles of weight one at each point in $S$.  Now by \cite[Proposition 2.1]{C06} (see Proposition 2.3), we have that:

$$3 = \gamma_u \|T\| \geq \sum_{i=1}^{4} \nu(T, q_i) + \sum_{i=1}^{4} \nu(T, p_i) > 4\alpha + 4\beta =\frac{4}{3}\alpha + \frac{8}{3} > 3.$$

\smallskip

\noindent This is a contradiction, thus we cannot have $m_1(S) \leq 3$.

Suppose $m_1(S) \geq 5$.  Let $L$ be the line such that $|S\cap L| \geq 5$.  If $L$ contains $\{p_i\}_{i=1}^{4}$ and one of the $q_i$, then $\Gamma_2$ is reducible (as regardless of which point $\Gamma_2$ omits, it still contains at least three points on $L$), and $L$ is a component which implies that $p_4 \in \Gamma_2$, which is impossible.  As the $q_i$ are in general position, $L$ contains three of the $p_i$ and two of the $q_i$.  If $p_1\in L$ then we have $L$ is a component of $\Gamma_1$ and at least one of $p_2$ or $p_3$ is on $L$, which is a component of $\Gamma_1$, and thus impossible as $p_2, p_3 \notin \Gamma_1$.  So $p_1\notin L$, but now $L$ contains $p_4$ and at least three points of $\Gamma_2$, so $L$ is a component of $\Gamma_2$, which means $p_4\in \Gamma_2$, another contradiction.  As the $q_i$ are in general position, this covers all the possible ways that $m_1(S)\geq5$.

So if there is $p_4 \in E_{\beta}^+(T) \backslash \Gamma_2$, it must be the case that $m_1(S) = 4$.  So there is a line $L$ containing exactly four points of $S$.  This decomposes into a few more cases depending on what four points the line $L$ contains.  The first and easiest is if $L$ contains $\{p_i\}_{i=1}^{4}$ (which means that none of the $q_i$ lie on $L$ as $m_1(S) = 4$).  Then consider the current

$$R = \frac{5\alpha ' - 2}{5\alpha '} [L] + \frac{2}{5\alpha '} T .$$ 

\smallskip

 Routine calculations show that $\|R\| = 1$, $\nu(R, p_i) > \frac{2}{5}$, and $\nu(R, q_i) > \frac{2}{5}$, so by  \cite[Proposition 1.2]{C06}, we have that there is a conic containing at least seven points of S, which means $L$ is a component of this conic, which implies that at least three of the $q_i$ are collinear as $L$ cannot contain more than four points, which is a contradiction.

\smallskip

We will now assume that $m_1(A) \leq 3$, and consider the remaining cases.  Then later we will consider them for when $m_1(A) = 4$.

\smallskip

If $L$ contains three $p_i$ and one $q_i$ then note that since $m_1(A)\leq 3$ it must be the case that $p_4\in L$.  Suppose that the four points not on $L$ are not in general position so there is a line, say $L_1$ containing  three of the points not on $L$, and they must be two $q_i$ and one $p_i$ (as the three $q_i$ not on $L$ are in general position), and $L\cap L_1 \cap A = \emptyset $ as $m_1(A) \leq 3$.  Noting that $| \Gamma_2 \cap(L\cup L_1)| \geq 5$, one of $L$ or $L_1$ is a component of $\Gamma_2$ by Bezout's theorem.  As $L$ contains $p_4$, it must be the case that $L_1$ is a component of $\Gamma_2$.   But since $L_1$ contains only three points of $\Gamma_2$, and at least two points of $\Gamma_2$ are on $L$, it must be the case that $\Gamma_2 = L\cup L_1$, but this means $p_4\in \Gamma_2$, which is a contradiction.  So the four points not on $L$ must be in general position.

\smallskip

  Note that since the four points off of $L$ must be in general position, and $L$ contains one of the $q_i$ and three of the $p_i$, we have satisfied all of the hypotheses of Lemma 3.2, and thus $\|T\| \neq 1$, which is a contradiction.

If $L$ contains two $p_i$ and two $q_i$, we let $B$ be the four point set consisting of the two $p_i$ and two $q_i$ not contained on $L$.  Since $m_1(A) \leq 3$, it must be the case that $p_4\in L$, and that $m_1(B) \leq 3$.  If $m_1(B) = 3$ then we can argue as we did above to get that $\Gamma_2$ is reducible, and it contains $L$ as a component, but then $p_4\in \Gamma_2$, which is impossible.  So it must be the case that $m_1(B) =2$ and again we can apply Lemma 3.2 to get a contradiction.  This finishes the case where $L$ contains two $p_i$ and two $q_i$, and also finishes the case $m_1(S) = 4$ when $m_1(A)\leq 3$.

 So far we have shown that if there is in fact a point $p_4\in E_{\beta}^+(T) \backslash \Gamma_2$, then it must be the case that $m_1(S) = 4 = m_1(A)$.  It only remains to consider the cases where $L$ contains one $q_i$ and three $p_i$ or two $q_i$ and two $p_i$.  We will first consider when $L$ contains three $p_i$, and let $B = S\backslash (S\cap L)$, noting that $m_1(B) < 4$ as the $q_i$ are in general position.  If $m_1(B)=2$ then by Lemma 3.2, $\|T\| > 1$, a contradiction.

 Thus $m_1(B) = 3$ and  then $m_2(S) = 7$.  After reindexing (if necessary) say that $p_1\in L$.  Let $C = L\cup L_1$ where $L_1$ contains the three collinear points in $B$ (noting that $L_1$ contains two $q_i$ and one $p_i$, and say $q_4$ is the point of$B$ not on $L_1$).  We will show that $C$ is the desired conic satisfying the conclusion of the proposition.  If not, assume for contradiction there exists $p_5 \in E_{\beta}^+ (T)\backslash C$.  If $L\cap L_1\cap S = \emptyset $, then set $A' = S\backslash \{p_1\}$ and $S' = A'\cup \{p_5\}$.  So note that $|A'| = 7$, $m_1(A') = 3$, $m_2(A') = 6$ (since if $m_2(A')=7$, either all four points in $B$ are collinear or one point of $B$ is on $L$ and neither of those can happen), $|A'\cap C | = 6$, $p_5\notin A'\cup C$, and $m_1(S') \leq 4$.  If $m_1(S') = 3$ then we can use \cite[Propositions 2.1 and 2.4.(i)]{C06} as before to get a contradiction.  If $m_1(S') = 4$ then since $p_5\notin C$, there is a line $L_2$ containing $p_5$ and three other points from $A'$.  By construction, $L_2$ must contain $q_4$ as well as one point of $L\cap S'$ and one point of $L_1 \cap S'$. However, $L_2$ contains at least one $q_i$ and $m_1(S' \backslash L_2) = 2$ so we can apply lemma 3.2 and thus $\|T\|>1$.  If $L\cap L_1\cap S \neq \emptyset$ then the intersection must be one of the points  contained on $L$, since otherwise if the intersection was a point on $L_1$, then $|L\cap S| = 5$, a contradiction.  Further, it must be one of the $p_i$,w.l.o.g. say $p_i =p_2$, as the $q_i$ are in general position.  We set $A' = S\backslash \{p_2\}$, and argue the same way to get a contradiction.  We have shown that if $m_1(A) = m_1(S) = 4$ and there is a line containing three of the $p_i$ and one $q_i$, then there can be no such $p_5$ and $C$ is the desired conic that satisfies the conclusion.

Finally we consider when $L$ contains two $p_i$, two $q_i$, and $m_1(A) = 4$.  Again we let $B = S\backslash S\cap L$ and note that $m_1(B)\neq4$ or else we get that $\Gamma_1 = \Gamma_2$.  Furthermore, if $m_1(B) = 2$, we can apply Lemma 3.2 to get a contradiction.  Our only remaining consideration is when $m_1(B) =3$.  Let $L_1$ be the line containing three points from $B$.  We will re-index our points so that $\{q_1, q_2, p_1, p_4\} \in L$ and $B = \{q_3, q_4, p_2, p_3\}$.

Let $C=L \cup L_1$, and again we will show this is the desired conic.  Suppose for contradiction that $p_5\in E_{\beta}^+(T)\backslash C$.  Assume $L\cap L_1 \cap S= \emptyset$.  Let $A' = S\backslash \{p_1\}$ (recalling $p_1 \in L$), $S' = A\cup \{p_5\}$, and note that $|A'| = 7$, $m_1(A') = 3$, $m_2(A') = 6$, $|A'\cap C | = 6$, $p_5\notin A'\cup C$, and $m_1(S') \leq 4$.  If $m_1(S') = 3$ then we can use \cite[Propositions 2.1 and 2.4.(i)]{C06} as before to get a contradiction.  If $m_1(S') = 4$ then since $p_5\notin C$, there is a line containing $p_5$ and three other points from $A'$, but now we argue as before using lemma 3.2 to reach a contradiction.  If instead $L \cap L_1 = \{p_i\}$, then note it must be some $p_i\in L$ (otherwise $m_1(S) >4$),  we set $A' = S\backslash \{p_i\}$ and the same argument shows that $C$ is the desired conic.

Suppose $L\cap L_1 = \{q_i\}$, and w.l.o.g. say that point is $q_i = q_1$.  Then $C = L\cup L_1$ omits $q_k\in B$,  (as the $q_i$ are in general position), say that omitted point is $q_4$.  We will let $L_2$ be the line that contains $q_4$ and $p_5$.  If $L_2\cap C \cap S = \emptyset$, then we can set $B' = \{q_3, q_4, p_3, p_5\}$, note that $m_1(B') = 2$, and apply Lemma 3.2 using $L$ and $B'$ to get a contradiction.  If $L_2$ hits exactly one point on $L\cap S'$ and no points on $L_1\cap S'$, then again we can let $B' = \{q_3, q_4, p_3, p_5\}$ and again use Lemma 3.2.  If $L_2$ hits exactly one point on $L_1\cap S'$ and no points on $L\cap S'$, then we can let $B' = \{q_2, q_4, p_1, p_5\}$ and again use Lemma 3.2.  If $L_2$ hits two points on $C\cap S'$, then note at least one of those two points must be a $p_i$ (as the $q_i$ are in general position) w.l.o.g. say it is $p_1$ on $L$, and we can set $B' = \{q_2, q_4, p_4, p_5\}$, again $m_1(B') = 2$.  Now using $L_1$, which contains $\{q_1, q_3, p_2, p_3\}$ (i.e. two $q_i$ and two $p_i$) and $B'$, we argue as before using Lemma 3.2 to get a contradiction.  This resolves the case of $L$ containing two $p_i$ and two $q_i$, the case of $m_1(A) = m_1(S) = 4$, and thus we have finished the proof.

\end{proof}

\begin{Proposition}\label{P:32}  Let $\{ q_{i} \}_{i=1}^{4}$ be points in $\mathbb{P} ^{2}$ such that $q_1, q_2, q_3$ lie on a line $L_1$ and $q_4$ does not fall on $L_1$.  In addition, $\nu (T, q_{i}) \geq \alpha > 2/5$.  Let $\beta = \frac{2}{3} (1- \alpha)$.  Then there exists a conic $C$ (possibly reducible) such that $| E_{\beta }^+ (T) \backslash C | \leq 1$.
\end{Proposition}

\begin{proof} Let $\{ q_{i} \}_{i=1}^{4}$, be as described in the assumptions, and let $p_1 \in E_{\beta}^+ (T)\backslash L_1$, with $p_1 \neq q_4$ (noting that if no such $p_1$ exists then we are done).  We will let $l_1$ be the line that connects $p_1$ and $q_4$ and let $\Gamma_1 = L_1 \cup l_1$.  Now there exist points $p_2, p_3\in E_{\beta}^+ (T) \backslash \Gamma_1$, else we are done.  Before moving on, we will show that we can assume that $m_1(\{p_1, p_2, p_3\}) =2$.  For suppose that all three $p_i$ lie on a line, say $l_2$, then $L_1\cup l_2$ gives us a conic containing six of the seven points.  Then there is a $p_4 \in E_{\beta}^+ (T) \backslash (L_1\cup l_2)$.  If $p_4 \notin l_1$ then note $\{p_1, p_2, p_4\}$ are in general position.  If $p_4\in l_1$, then note $\{p_2, p_3, p_4\}$ are in general position.  Either way, we will reindex the set and call the points $\{p_1, p_2, p_3\}$ where $p_1$ is the point on $\Gamma_1$.  Let $\alpha '$ be such that $\alpha > \alpha ' > 2/5$ and $\nu (T, p_{i}) > \frac{2}{3} (1- \alpha ') > \beta$ and let $L_{jk}$ be the containing $p_j$ and $p_k$.  Define a current R as follows:

$$R = \frac{5\alpha ' - 2}{15\alpha '}\sum_{1\leq j<k\leq 3} [L_{jk}] + \frac{2}{5\alpha '} T $$ 

\smallskip

\noindent and note $\|{R}\| = 1 $.  We have the following inequalities:

$$\nu (R, q_i) > \frac{2}{5\alpha '} \alpha > \frac{2}{5} $$

\noindent and

$$ \nu (R, p_{i}) > \frac{10\alpha ' - 4}{15\alpha '} + \frac{4-4\alpha '}{15\alpha '} = \frac{2}{5}. $$

\smallskip

Thus by \cite[Theorem 1.2]{C06} (see Theorem 2.2), there is a conic $\Gamma_2$ containing at least six of the $\{ q_{i} \}_{i=1}^{4} \cup \{ p_{i} \}_{i=1}^{3}$.  As $\Gamma_1$ is uniquely defined by the $q_i$ and $p_1$, $\Gamma_2$ must omit one of the seven points, and the point omitted must be one of the $q_i$ or $p_1$, else $\Gamma_1 = \Gamma_2$, which means one or both of $p_2, p_3$ would be on $\Gamma_1$, which is a contradiction. If $\Gamma_2$ satisfies the conclusion, then we are done.  So suppose $\Gamma_2$ does not satisfy the conclusion of our proposition, and then there is $p_4\in  E_{\beta}^+ (T)\backslash \Gamma_2$.

 We will let $A = \{q_i\}_{i=1}^{4} \cup \{p_i\}_{i=1}^{3} $, and we will note that $|A| = 7$, $m_2 (A) = 6$, $|A\cap \Gamma_2 | = 6$ and $p_4\notin A\cup \Gamma_2$.  We will make use of these observations shortly.  Define $S = A\cup\{p_4\}$.  We now consider the following possibilities for $S$: $m_1(S) \leq 3$, $m_1(S) = 4$, and $m_1(S)\geq 5$.

 Suppose $m_1(S)\leq 3$.  Then this means that $m_1(A)\leq 3$ and so we can apply \cite[Proposition 2.4.(i)]{C06}, i.e. there exists $u\in PSH(\mathbb{C}^2)$ such that $\gamma_u = 3$, $u$ is locally bounded outside of a finite set, and $u$ has logarithmic poles of weight one at each point in $S$.  Now by \cite[Proposition 2.1]{C06}, we have that:

$$3 = \gamma_u \|T\| \geq \sum \nu(T, q_i) + \sum \nu(T, p_i) > 4\alpha + 4\beta =\frac{4}{3}\alpha + \frac{8}{3} > 3.$$

\smallskip

\noindent This is a contradiction, thus $m_1(S) > 3$.

Suppose $m_1(S) \geq 5$.  Note that by how the points in $A$ are constructed, it is the case that $m_1(A)\leq 4$, and since $m_1(S)\geq 5$, this means $m_1(A) = 4$, and as the $p_i$ are in general position, the only way that $m_1(A) = 4$ is if there is a line containing $\{q_4, p_2, p_3, q_i\}$ for some $i=1,2,3$.  Then there is a line $L$ containing at least five points, and it must be the previously mentioned line with $p_4$ on it as well.  However, regardless of what point is omitted from $\Gamma_2$, $L$ is a component of $\Gamma_2$ which means $p_4\in \Gamma_2$, which is a contradiction.  Thus $m_1(S) < 5$.

It must be the case that $m_1(S) = 4$, and now we begin our battle with this situation.  As before, we will note that this breaks into cases depending on what points lie on the the line that contains four points.  As the $p_1, p_2,$ and $p_3$ are not collinear, we cannot have all four $p_i$ on a line, so that removed that case instantly.

\textit{Case} 1:  Suppose $L$ contains three $p_i$ and one $q_i$.  Suppose that $q_i = q_4$.  If $p_1 \in L$, the conic $\Gamma_3 := L\cup L_1 = \Gamma_1$, which is impossible as one of the other two $p_i$ on $L$ will be either $p_2$ or $p_3$, and $p_2, p_3 \notin \Gamma_1$.  So it must be that the $p_i$ are $p_2, p_3,$ and $p_4$.  Note $| \Gamma_3\cap \Gamma_2 | \geq 5$, and that any subset of five points from $\{q_1, q_2, q_3, q_4, p_2, p_3\}$ uniquely defines $\Gamma_3$ so it must be the case that $\Gamma_2 = \Gamma_3$, which means $p_4 \in \Gamma_2$, which is a contradiction.  Thus $q_i \neq q_4$.

 So $L$ contains a $q_i \neq q_4$, say $L$ contains $q_1$ (reindexing if necessary).  Once again note that $p_4$ must be one of the points on $L$ as otherwise we would have $p_1, p_2, p_3$ collinear.  Let $B = \{ q_2, q_3, q_4, p_i\}$ be the four points off $L$.  If $m_1(B) = 2$, then we are done as Lemma 3.2 gives us a contradiction.  So it must be the case that $m_1(B) \geq 3$, and as $q_4 \notin L_1$, we have $m_1(B)=3$.  Since $p_i\notin L_1$ (because $p_i \neq p_4$), we have a line, $L_2$,that contains $\{p_i, q_4, q_i\}$  (w.l.o.g. say $q_2$). Let $C := L\cup L_2$, we will show $C$ is the desired conic.  For contradiction suppose there is $p_5\in E_{\beta}^+ (T)\backslash C$.  Note that if $L_2 \cap L \cap A = \emptyset$, $C$ is uniquely determined by any five points of $\{q_1. q_2, q_4, p_1, p_2, p_3\}$.  Also note that $| \Gamma_2\cap C| \geq 5$, so again we can argue that $\Gamma_2 = C$, but again this means $p_4 \in \Gamma_2$, a contradiction.  If instead $L_2 \cap L \cap A = \{p_2\}$ (reindex if necessary), then we consider the set $A' = S\backslash \{p_2\}$ and $S' = A'\cup \{p_5\}$.  Note $|A'|=7$, $m_1(A') = 3$, $m_2(A') = 6$, $|A'\cap C | = 6$, and $p_5\notin A'\cup C$.  Let $L_3$ be the line containing $p_5$ and $q_3$.  If $L_3 = L_1$, i.e. $p_5\in L_1$, then note the line $L_1$ and $\{p_1,p_3,p_4,q_4\}$ satisfy the assumptions of Lemma 3.2, giving us a contradiction.  If $p_5\notin L_1$ and $|L_3 \cap C \cap S'| \leq 1$ then $m_1(S')\leq 3$ and we can argue using \cite[Propositions 2.1 and 2.4.(i)]{C06} to get a contradiction.  Finally if  $|L_3 \cap C\cap S'| = 2$, then $L_3$ contains one point of $L_2\cap S'$ and one point of $L\cap S'$.  But now note that $m_1(S' \backslash (S'\cap L_3)) = 2$, so those four points and $L_3$ satisfy the assumptions of Lemma 3.2, and again we get a contradiction.

\textit{Case} 2:  Now suppose $L$ contains three $q_i$ and one $p_i$.  Actually it must be the case that $L = L_1$ and $p_4\in L_1$, as no other $p_i$ can be on $L_1$.  If $m_2(S) = 6$ then the four points not on $L$ are in general position, and thus by Lemma 3.2, we have a contradiction.  Since $m_2(A) = 6$, $m_2(S)\leq 7$, so it must be the case that $m_2(S) = 7$.  Let $B$ be the set containing the four points not on $L$, and it must be that $m_1(B) =3$ (else $m_2(S)\neq 7$).  Since $m_1(B) = 3$, $p_1, p_2, p_3$ cannot be collinear, and $p_2, p_3 \notin \Gamma_1$, there is a line, say $L_2$ containing $\{p_2, p_3, q_4\}$. However it now follows that $m_1(A) = 4$ since if $m_1(A) =3$, then we would get that $\Gamma_2 = L_2 \cup L$ which means $p_4\in \Gamma_2$, a contradiction.  So there is a line containing $p_2, p_3, q_4$ and one of the $q_i$ on $L$ (as this is the only way we can have $m_1(A) = 4$), and that line is in fact $L_2$.  Let $C = L\cup L_2 $, and note there must be a $p_5\in E_{\beta}^+ (T)\backslash C$, otherwise we are done.  Let $L_3$ be the line containing $p_1,p_5$.  If $L_3 \cap C\cap S = \emptyset$, then note $m_1(\{p_1,p_2,p_5, q_4\}) = 2$, so those four points and the line $L$ satisfy the hypotheses of Lemma 3.2.  If $L_3 \cap C\cap S = \{p_i\}$, then we can assume w.l.o.g. that $p_i$ is $p_2$ on $L_2$, and now the points  $p_1, p_3, p_5, q_4$ are in general position and none of the fall on $L$, so again we can apply Lemma 3.2 to get a contradiction.  If instead  $L_3 \cap C\cap S = \{q_i\}$, say $q_1$ on $L$, then note $p_1, p_2, p_5, q_4$ are in general position and off $L$, so again we can use 3.2.  A similar argument holds if $q_i$ falls instead on $L_2$ or on the intersection $L\cap L_2$.  If  $|L_3 \cap C\cap S|=2$ and at least one of the two points is a $p_i$, we can argue as we did above.  If both points are $q_i$, one must be $q_4$ on $L_2$ and say the other is $q_1$ on $L$, however this is the same configuration that we resolved in Lemma \ref{L:33}, and thus this situation cannot happen either.  We have have proven that there cannot exist a point $p_5$, and thus $C$ is the desired conic, resolving the case when our line $L$ contains three $q_i$ and one $p_i$.

\textit{Case} 3:  We now move on to our last situation, that the line $L$ contains two $p_i$ and two $q_i$.  As $m_1(S) = 4$, one of the $q_i$ is $q_4$, and the other is one of the three $q_i$ on $L_1$, w.l.o.g., say $q_1$, and say the other points are $p_2$ and $p_3$.  Let $B$ once again be the four points off of $L$, so $B = \{q_2, q_3, p_1, p_4\}$ and either $m_1(B) = 2$ or $m_1(B)=3$ (if $m_1(B) = 4$, this means thats $\Gamma_1 = \Gamma_2$, which is impossible).  If $m_1(B) = 2$, then by Lemma 3.2, we have a contradiction.  If $m_1(B) = 3$, and we have one of the $p_i$ on $L_1$ and we are back in case two as now $L_1$ contains three $q_i$ and one $p_i$, which we have already argued.  So let $L_2$ be the line containing three points of $B$ and note that it must be both $p_i$ and one of the $q_i$ on $L_1$, say $q_2$.  Let $C := L\cup L_2$, and we will show that $C$ s the desired conic.  Suppose for contradiction that there is $p_5\in E_{\beta}^+(T)\backslash C$.  If $p_5\in L_1$, and if $L\cap L_2 \cap S=\emptyset$ then note we can use Lemma 3.2 with $p_1, p_2, p_4, q_4$ as they are in general position, and $L_1$, giving a contradiction.  If $p_5\in L_1$, and if $L\cap L_2 \cap S=\{p_i\}$, then we  can use Lemma 3.2 again, but using the four point off of $L_1$ that omits $\{p_i\}$.  If $p_5\in L_1$, and if $L\cap L_2 \cap S=\{q_i\}$, then we we can apply Lemma \ref{L:33} to get a contradiction.  Thus $p_5\notin L_1$, and then let $L_3$ be the line containing $p_5$ and $q_3$.  If $L_2\cap L_3 \cap B\neq \emptyset$ then it must that the intersection is one of the $p_i$, say $p_1$, for if the intersection is $q_2$, then that forces $p_5\in L_1$.  But now note that $m_1(\{q_2, q_3, p_4, p_5\}) = 2$, and all of the points are off $L$, so we can apply Lemma 3.2, and get a contradiction.  If $L_2\cap L_3 \cap B= \emptyset$, then the same argument holds.  Since $p_5$ can neither be on $L_1$ or off $L_1$, no such point can exist, and thus $C$ is the desired conic that satisfies the conclusion.  This resolves the third case, which finishes the $m_1(S) = 4$ case, and thus, the proof.
\end{proof}

\begin{Proposition}\label{P:33}  Let $\{ q_{i} \}_{i=1}^{4}$ be points in $\mathbb{P} ^{2}$ such that all four points are collinear and $\nu (T, q_{i}) \geq \alpha > 2/5$.  Let $\beta = \frac{2}{3} (1- \alpha)$.  Then there exists a conic $C$ such that $| E_{\beta }^+ (T) \backslash C | \leq 1$.
\end{Proposition}

\begin{proof} Let $L$ be the line containing the $q_i$, and suppose $| E_{\beta}^+ (T) \backslash L | > 1 $, (otherwise we are done), so there exist points $p_1, p_2 \in E_{\beta}^+ (T)$ not on $L$, and let $L_{12}$ be the line they lie on.  We want to generate four points of $E_{\beta}^+(T)$ that do not lie on $L$ such that no three are collinear.  If the conic $L\cup L_{12}$ does not satisfy the conclusion then we can find two more point $p_3, p_4 \in E_{\beta}^+(T) $ that do not lie on our conic, and let $L_{34}$ be the line containing these new points.  If the four $p_i$ are in general position then we are done, otherwise $L_{34}$ contains three of the $p_i$, after reindexing, say it contains $p_1, p_3, p_4$.  If the conic $L \cup L_{34}$ does not satisfy the conclusion then we can find a point $p_5 \in E_{\beta}^+(T)$ that is not on the new conic.  If $p_5$ does not fall on $L_{2k}$ for $k=3,4$, then take $p_2, p_3, p_4,p_5$ as our four points in general position.  If $p_5$ falls on $L_{2k}$, say w.l.o.g. $L_{23}$, then we take $p_1, p_2, p_4, p_5$ as our four points in general position.  We will reindex to the points to be $p_1, p_2, p_3, p_4$.

By Siu's decomposition theorem \cite{Siu74} we have that 

$$T = a [L ] + R,$$
\smallskip

\noindent where $a$ is the generic Lelong number of $T$ along $L$.  Note that $\| R \| = 1 - a $ and $\nu (R, q_i) \geq \alpha - a$.  Let $\alpha ' \in (\frac{2}{5},  \alpha )$ be such that $\nu (T, p_i) = \nu (R,p_i) > \frac{2}{3}(1-\alpha ') > \beta $ for $i=1,2,3,4$.   Proposition 2.5 shows that there exists a current $R'$ such that $ \|R'\| = 1 - a $, $R'$ is smooth where $R$ has Lelong number $0$, and $\nu (R', q_i) > \alpha ' - a $.  By \cite[Corollary 2.10]{D93}, $R' \wedge [L] $ is a well defined measure.  Now we have

$$ 1-a = \int_{\mathbb{P}^{2 }} R' \wedge [L] \geq \sum_{i=1}^{4} \nu (R'\wedge [L] , q_i ) \geq \sum_{i=1}^{4} \nu (R',q_i) \nu ([L], q_i) >4\alpha ' - 4 a, $$

\smallskip

\noindent where the second inequality follows from \cite[Corollary 5.10]{D93} and the final inequality follows as $\nu([L], q_i) = 1$.  So we have that $a > \frac{4\alpha ' - 1}{3}$.

 Define a new current:

$$S = \frac{R}{1- a} $$

\smallskip

 \noindent and note $\| S \| = 1$.  Now we have:

$$\nu (S, p_i) > \frac{2}{3} \frac{1-\alpha'}{1-a} > \frac{2 - 2\alpha'}{4-4\alpha'} = \frac{1}{2} ,\; i=1,2,3,4. $$

\smallskip

Coman's result, \cite[Theorem 1.1]{C06} shows that $m_1(\{p_1, p_2, p_3, p_4\}) \geq 3$ which implies that at least three of the $p_i$ are collinear which is a contradiction as we constructed them to be in general position. 
\end{proof}

Theorem 1.1 now follows by combining the previous three propositions.

\vspace{10pt}

The following examples will show the necessity of allowing for $|E_{\beta}^+(T) \backslash C| = 1$ since we can have $E_{\beta }^+ (T) \not \subset C$ for all conics $C$.  Also we will see that that $\beta = \frac{2}{3} (1-\alpha)$ is sharp for this property, and that the result fails if we have less than four point with  ``large" Lelong number.

\vspace{10pt}

\noindent \textbf{Example 3.7.}  Let $L_i$, $i = 1,2,3,4$ be complex lines such that no three intersect at the same point.  Define a current $T =\frac{1}{4} \sum_{i=1}^{4} [L_i]$ and let $\alpha = \frac{1}{2}$.  Note that there are six points with Lelong number $\frac{1}{2}$, so we have satisfied the assumptions of the main theorem, and note that $\beta = \frac{1}{3}$.  As each $L_i$ contains exactly three points of $E_{1/3}^+(T)$, and any pair of the $L_i$ contains exactly five of the points in $E_{1/3}^+(T)$ it follows that for any conic satisfying the result of the corollary, we have one point in $E_{1/3}^+(T)$ not on the conic.

\smallskip

\noindent \textbf{Example 3.8.}  Let $L_i$, $i = 1,2,3,$ be complex lines such that they do not intersect at the same point.  Let $L_1\cap L_2 = \{q_3\}$ , $L_1\cap L_3 = \{q_2\}$ , $L_3\cap L_2 = \{q_1\}$.  Let $q_4\notin L_1\cup L_2 \cup L_3$ and let $L_4, L_5, L_6$ be the lines connecting $q_4$ with $q_1, q_2, q_3$ respectively.  Also $L_4 \cap L_1 =\{p_1\}$, $L_5 \cap L_2 =\{p_2\}$, $L_6 \cap L_3 =\{p_3\}$.  Note that $m_1(\{p_1, p_2, p_3, q_4\}) = 2$. Finally define a current $T =\frac{1}{6} \sum_{i=1}^{6} [L_i]$.  Note that $\nu(T, q_i) = \frac{1}{2}$ and $\nu(T, p_i) = \frac{1}{3}$.  Let $\alpha = \frac{1}{2}$, and note that since $\beta = \frac{1}{3}$, we have that $E_{\beta}^+(T) = \{q_1, q_2, q_3, q_4\}$ which can clearly be contained in a conic, but $E_{\beta}(T) = \{q_1, q_2, q_3, q_4, p_1, p_2, p_3\}$, and $m_2(E_{\beta}(T)) = 5$.

\smallskip

\noindent \textbf{Example 3.9.}  Let $L_i$, $i = 1,2,3,$ be complex lines such that they do not intersect at the same point.  Let $L_1\cap L_2 = \{q_3\}$ , $L_1\cap L_3 = \{q_2\}$ , $L_3\cap L_2 = \{q_1\}$ and define a current $T =\frac{1}{3} \sum_{i=1}^{3} [L_i]$.  Note that $\nu(T, q_i) = \frac{2}{3}$ so if we set $\alpha = \frac{2}{3}$, then we have exactly three points with Lelong number at least $\alpha$, and $\beta = \frac{2}{9}$, thus then $E_{\beta}^+(T)$ contains all three lines, and $|E_{\beta}^+(T) \backslash C| = \infty$ for all conics $C$.

\smallskip

It is even interesting to note that the result fails in the special case where we have only three points with large Lelong number that are collinear.

\smallskip

\noindent \textbf{Example 3.10.}  Let $\{q_i\}_{i=1}^{3} \cup \{p_i\}_{i=1}^{6}$ be points and $\{L_i\}_{i=1}^{3}$ be lines such that $\{q_1, q_2, q_3, p_1\}\in L_1$, $\{q_1, p_2, p_3, p_6\}\in L_2$, and $\{q_3, p_4, p_5, p_6\}\in L_3$.  Also let $\{l_i\}_{i=1}^{4}$ be lines such that $\{q_2, p_2, p_4\}\in l_1$, $\{q_2, p_3, p_5\}\in l_2$, $\{p_1, p_2, p_5\}\in l_3$, and $\{p_1, p_3, p_4\}\in l_4$.  Let $\alpha = \frac{9}{20}$, which means $\beta = \frac{11}{30}$.  We will instead write them as $\alpha = \frac{81}{180}$, $\beta = \frac{66}{180}$.  We now consider the current given by

$$T = \frac{46}{180} [L_1] + \frac{37}{180} \sum_{i=2}^{3} [L_i] +\frac{19}{180} \sum_{i=1}^{2} [l_i] +\frac{11}{180} \sum_{i=3}^{4} [l_i]$$

 \noindent and note $\|T\| = 1$.  Now calculating the Lelong numbers at each points we have:

\begin{eqnarray*}
&&\nu(T, q_1) = \frac{83}{180}\quad \nu(T,q_2) = \frac{84}{180},\quad \nu(T, q_3) = \frac{83}{180} \\
&&\nu(T, p_1) = \frac{68}{180}\quad \nu(T,p_2) = \frac{67}{180},\quad \nu(T, p_3) = \frac{67}{180} \\
&&\nu(T, p_4) = \frac{67}{180}\quad \nu(T,p_5) = \frac{67}{180},\quad \nu(T, p_6) = \frac{74}{180}
\end{eqnarray*}

\smallskip

\noindent and note $\nu(T,q_i) > \alpha$ for $i=1,2,3$ and $\alpha > \nu(T, p_i) > \beta$ for $i=1,\dots ,6$.  So we have exactly three points where $T$ has Lelong number larger than $\alpha$, and these are collinear.  However there are no conics that can contain more than seven of the nine points, i.e. $|E_{\beta}^+(T)\backslash C| \geq 2$ for all conics $C$.

\end{document}